\documentclass[a4paper,11pt]{amsart}
\usepackage[all,arc]{xy}
\usepackage[all]{xy}
\usepackage{latexsym}
\usepackage{amssymb}
\usepackage{amsmath}
\newtheorem{definition}{Definition}[section]

\newtheorem{lemma}[definition]{Lemma}
\newtheorem{prop}[definition]{Proposition}
\newtheorem{theorem}[definition]{Theorem}
\newtheorem{cor}[definition]{Corollary}
\newtheorem{conj}[definition]{Conjecture}

\newtheorem{fact}[definition]{Fact}

\newcommand*{\ms}{\models}

\newcommand*{\bsm}{\left(\begin{smallmatrix}}
\newcommand*{\esm}{\end{smallmatrix}\right)}
\newcommand*{\bp}{\begin{pmatrix}}
\newcommand*{\ep}{\end{pmatrix}}

\newcommand*{\fty}{\infty}

\newcommand*{\wg}{\wedge}
\newcommand*{\wh}{\widehat}
\newcommand*{\wt}{\widetilde}

\newcommand*{\la}{\longrightarrow}
\newcommand*{\xl}{\xleftarrow}

\newcommand*{\mB}{\mathcal{B}}

\newcommand*{\N}{\mathbb{N}}

\newcommand*{\al}{\alpha}
\newcommand*{\be}{\beta}
\newcommand*{\del}{\delta}
\newcommand*{\eps}{\varepsilon}
\newcommand*{\ga}{\gamma}

\newcommand*{\lam}{\lambda}
\newcommand*{\Lam}{\Lambda}

\renewcommand*{\phi}{\varphi}

\begin{document}

\footskip=30pt

\date{}

\title{Ringel's conjecture for domestic string algebras.}

\author{Gena Puninski}

\address{Department of Mechanics and Mathematics, Belarusian State University, Praspekt Nezalezhnosti 4,
Minsk 220030, Belarus}
\email{punins@mail.ru}

\author{Mike Prest}

\address{School of Mathematics, Alan Turing Building, University of Manchester, Manchester M13 9PL, UK}
\email{mprest@manchester.ac.uk}

\thanks{The first author is grateful for EPSRC support on Grant EP/K022490/1 for his one month visit to Manchester University during which
this paper was written. He also is indebted to the University for an encouraging scientific environment. Both authors
are greatly indebted to a referee for many useful suggestions.}



\begin{abstract}
We classify indecomposable pure injective modules over domestic string algebras, verifying Ringel's
conjecture on the structure of such modules.
\end{abstract}

\maketitle

\pagestyle{plain}

\section{Introduction}\label{S-Int}

In the realm of finite dimensional algebras, pure injective modules may be defined as direct summands of direct
products of finite dimensional modules. Even when the structure of finite dimensional modules is known (say, the
algebra is tame) the effect of this construction is understood in just a few cases.

For instance (see \cite{Pun04}) superdecomposable pure injective modules can occur over some tame (string,
non-domestic) algebras. If we are interested in indecomposable pure injective modules only, then a complete
classification is known for tame hereditary finite dimensional algebras (see \cite{Pre98} or \cite{Rin98}) but in
just a few other examples (see \cite{B-P}).

In 1995 Ringel \cite{Rin95} constructed examples of pure injective modules over string algebras corresponding to
some infinite strings and (in \cite{Rin00}) conjectured that these modules, the finite dimensional modules, together
with the infinite dimensional band modules (i.e. Pr\"ufer, adic and generic) is a complete list of the
indecomposable pure injective modules over domestic string algebras. Over the past 15 years there have been
persistent attempts, starting from \cite{B-P}, to settle this conjecture, most recently in the (as yet largely
unpublished) thesis of Richard Harland \cite{Har} and by the first author \cite{Pun14}, who completed the case of
1-domestic string algebras.

In this paper we will complete the proof of Ringel's conjecture by reducing it to the (already known) 1-domestic
case. It was the first author who saw how to use a portion of Harland's thesis to effect this reduction, and we will
use this opportunity to include some of Harland's arguments, though his proof of one particularly important theorem
is not included because of its highly combinatorial nature and strong dependence on the approach and details of the
whole of his text.  We will, however, give a direct proof of that result in the particular case of domestic string
algebras, leaving that to the end of the paper so as not to break the flow of the argument.  We will provide each
statement borrowed from \cite{Har} with a precise reference.

The main drawback of the current situation is that the interested reader
should go through a couple of hundred pages of journal papers, preprints and a thesis, in order to recover the whole
logic of the proof. Of course there is need for a unified text with more accessible and conceptual proofs, but for
now we are satisfied with obtaining the result.  It may be that the combinatorial nature of string algebras prevents
there being a proof which is significantly easier and shorter.

We will also prove the non-existence of superdecomposable pure injective
modules over domestic string algebras as a consequence of our main result: every pure injective module over such an
algebra contains an indecomposable direct summand.  We will postpone a further consequence, the finiteness of
Krull--Gabriel dimension of domestic string algebras to a forthcoming paper.

\section{Basics}\label{S-bas}

We will be very brief on basic definitions, relying mostly on illustrative diagrams. One can find rigorous
definitions in \cite{B-R}, with almost no diagrams, or in Schr\"oer's thesis \cite{Sch97}, with a lot of them.

In this paper $A$ will denote a finite dimensional string algebra of infinite representation type over an
algebraically closed field $K$.  The assumption that $K$ be algebraically closed is for convenience and simplicity of
arguments when treating band modules; it is almost certainly not essential for any result. A \emph{string} algebra is
a special kind of a bound quiver algebra $KQ/I$ with monomial relations, in particular there are at most two ingoing
and two outgoing arrows at each vertex of $Q$ and also, for each ingoing arrow, there is at most one nonzero
composition with an outgoing arrow, and {\it vice versa}. For instance, the Kronecker algebra

$$
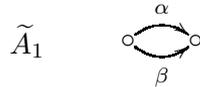
\wt A_1 \hspace{1cm}
\vcenter{%
\xymatrix@R=8pt@C=20pt{%
*+={\circ}\ar@/^.5pc/[r]^{\al}\ar@/_.5pc/[r]_{\be}&*+={\circ}\\
}
}
$$\label{A1}

\vspace{2mm}

\noindent is a string algebra which is hereditary. A non-hereditary example is the Gelfand--Ponomarev algebra
$G_{2,3}$, which is the path algebra of the following quiver

$$
G_{2,3} \hspace{1cm}
\vcenter{
\xymatrix@R=16pt@C=16pt{%
*+={\circ}\ar@(ul,dl)_{\alpha}\ar@(ur,dr)^{\beta} } }
$$

\vspace{2mm}

\noindent with relations $\al^2= \be^3= \al\be= \be\al=0$.

For another example one could take the algebra $R_1$ with the same quiver but with noncommutative relations:
$\al^2= \be^2= \al\be= 0$. Finally, let $\Lam_2$ be the following string algebra:

$$
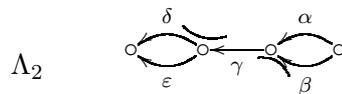
\Lam_2 \hspace{1cm}
\vcenter{
\xymatrix@R=20pt@C=20pt{%
*+={\circ}&*+={\circ}\ar@/_0.5pc/[l]_{\del}="del"\ar@/^0.5pc/[l]^{\eps}="eps"&*+={\circ}\ar[l]^{\ga}="ga"&
*+={\circ}\ar@/_0.5pc/[l]_{\al}="al"\ar@/^0.5pc/[l]^{\be}="be"\\
\ar@/_0.2pc/@{-}"del"+<6pt,-4pt>;"ga"+<-4pt,12pt>
\ar@/^0.2pc/@{-}"ga"+<8pt,4pt>;"be"+<-6pt,4pt>
}
}
$$

\vspace{2mm}

\noindent whose relations $\del\ga= \ga\be= 0$ are shown by solid curves. Note that we apply arrows from right to left
- in the path $\del\ga$ one goes first by $\ga$ and then by $\del$ - and our modules will be left modules over the
path algebra.

A \emph{letter} is an arrow (a \emph{direct} letter) or the formal inverse of an arrow (an \emph{inverse} letter).  A (finite) \emph{string} over $A$ is a walk through the quiver of $A$ (hence a finite word in direct and inverse
letters) such that there can be no cancellation (of a letter and its inverse) nor can any relation (or its inverse)
be met on the way. For instance $\al\be^{-1}$ is a string over each of the above algebras, which we will illustrate by
the following diagram:

$$
%\Lam_2 \hspace{1cm}
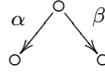
\vcenter{
\xymatrix@R=12pt@C=8pt{%
&*+={\circ}\ar[dl]_{\al}\ar[dr]^{\be}&\\
*+={\circ}&&*+={\circ}
}}
$$

\vspace{2mm}

\noindent (note that we draw direct arrows from the upper right to the lower left, and inverse arrows from the upper
left to the lower right).  By the definition of string algebra, there is at most one way to extend a string of length
$\geq 1$ to the right by a direct arrow, to the right by an inverse arrow, and similarly on the left.

In fact the above diagram represents not a string, but rather the corresponding string module (the circles correspond
to basis elements), which is finite dimensional and indecomposable.

If $u$ is a string then the corresponding string module, denoted $M(u)$, is indecomposable. According to \cite{B-R} all the
other indecomposable finite dimensional $A$-modules are \emph{band modules}, where a band is a walk which returns to
its starting point, contains both a direct and an inverse arrow, can be repeated twice and is not a proper power of another
walk. A typical example, see Figure \ref{figband}, is the (two layer) band module $M= M(C,\lam, 2)$ corresponding to
the band $C= \al\be^{-1}$ over $\widetilde{A_1}$ and $0\neq \lam\in K$. For instance $\be(z_2^2)= \lam z_1^2+ z_1^1$.

\begin{figure}[b]
$$
\vcenter{
\xymatrix@R=24pt@C=30pt{%
*+={\circ}\ar@{}+<-10pt,10pt>*{z_1^2}\ar@{.}[d]&
*+={\circ}\ar@{.}[d]\ar@/^0.9pc/[l]_{\alpha}\ar@{.}[d]\ar@{-->}@/_0.3pc/[dl]^(.3){\be}\ar@/_0.9pc/[l]_{\beta=\lambda}
\ar@{}+<10pt,10pt>*{z_2^2}\\
*+={\circ}\ar@{}+<-10pt,-10pt>*{z_1^1}&
*+={\circ}\ar@/^0.9pc/[l]^{\alpha}\ar@/_0.9pc/[l]^{\beta=\lambda}\ar@{}+<10pt,-10pt>*{z_2^1}
}
}
$$
\caption{}\label{figband}
\end{figure}

\vspace{2mm}

It is technically convenient to insist that a band should start at a point in the socle of the corresponding module:
so a \emph{band} will mean any string $C$ of the form $\al u \be^{-1}$, where $\al$ and $\be$ are different arrows
with target the same vertex of $Q$, such that $C$ is \emph{primitive}, i.e. $C\neq v^k$ for $k\geq 2$ and any string
$v$. Of course if $C$ is a band then so also is $C^{-1}= \be u^{-1}\al^{-1}$, as well as some cyclic permutations of
$C$. We assume that $A$ is of infinite representation type, so there are some
bands over $A$.

For example the algebras $\wt A_1$ and $R_1$ have essentially one band $\al\be^{-1}$, but $\Lam_2$ has two bands
$\al\be^{-1}$ and $\eps\del^{-1}$. There are infinitely many bands over $G_{2,3}$, for which
$\al\be^{-2}\al\be^{-2}\al\be^{-1}$ is one example.

One can extend the notion of a finite string to (1-sided or 2-sided) infinite string. For instance the 1-sided string
$(\al\be^{-1})^{\fty}$ is periodic, whereas the string $\be(\al\be^{-1})^{\fty}$ over $R_1$ is \emph{almost periodic},
but not periodic. Furthermore, the 2-sided string $^{\fty} (\be\al^{-1})\be (\al\be^{-1})^{\fty}$
(see Figure \ref{fig1}) over $R_1$ is \emph{biperiodic} (i.e. almost periodic on the right and on the left but not
periodic).

\begin{figure}[b]
$$
\vcenter{
\xymatrix@R=14pt@C=10pt{%
&&&&&&&*+={\circ}\ar[dl]_{\al}\ar[dr]^{\be}&&*+={\circ}\ar[dl]^{\al}\ar[dr]^{\be}&&\\
&&*+={\circ}\ar[dl]_{\be}\ar[dr]^{\al}&&*+={\circ}\ar[dl]^{\be}\ar[dr]^{\al}&&*+={\circ}\ar[dl]^{\be}&&*+={\circ}&&
*+={\circ}&\dots\\
\dots&*+={\circ}&&*+={\circ}&&*+={\bullet}&&&&&&
}
}
$$
\caption{}\label{fig1}
\end{figure}

\vspace{2mm}

A string algebra $A$ is said to be \emph{domestic} if, for every arrow $\al$, there exists at most one band over $A$
starting with $\al$, equivalently, \cite[Prop.~2]{Rin95}, there are just finitely many bands. (In fact this is a
specialization for string algebras (see \cite{Rin95}) of the general notion of domesticity for tame finite
dimensional algebras.) For example, $\al\be^{-1}$ is the unique (up to a cyclic permutation and inversion) band for
the algebras $\widetilde{A_1}$ and $R_1$, therefore we say that these algebras are 1-domestic. For $\Lam_2$ we have essentially
two bands, therefore this algebra is 2-domestic. Finally $G_{2,3}$ has infinitely many bands starting with $\al$,
hence this string algebra is not domestic.

By \cite[Prop.~2]{Rin95} a string algebra $A$ is domestic if and only if every 1-sided string over $A$ is almost
periodic. It follows that every 2-sided string over $A$ is either biperiodic or periodic.

Suppose that $u$ is an infinite string which is either 1-sided almost periodic on the right, or 2-sided almost
periodic on the right, but not totally periodic. It can be uniquely written in the form  $v l D^{\fty}$, where $D$ is
a primitive cycle and $l$ is a (direct or inverse) letter such that the string  $lD^{\fty}$ is no longer periodic; we allow $l$ to be empty, in which case $u$ is 1-sided periodic.  Following Ringel \cite{Rin95} we say that this string is \emph{expanding} if the last (meaning right-most) letter of $D$ is inverse. From the definition of string algebra it follows that $l$ is direct (or empty) and there is a repeatable ``shift" to the right of the word away from the middle - meaning the non-periodic part - which corresponds to an endomorphism of the
associated infinite-dimensional module. For instance the string shown on Figure \ref{fig1} is expanding with
$l= \be$ and $D= \al\be^{-1}$.

We say that $u$ is \emph{contracting} if the last letter of $D$ is direct, hence $l$ is an inverse letter (or empty) and there
is a corresponding shift/endomorphism towards the middle of the word. Corresponding `left' notions are defined by
considering the string $u^{-1}$. For instance the string in Figure \ref{fig1} is contracting on the left: one should
first flip it over and check this condition on the right.

To each 1-sided almost periodic (including periodic) string, and to each 2-sided biperiodic (that is, almost periodic
in each direction but not (totally) periodic) string Ringel \cite{Rin95} assigned an indecomposable pure injective
module which is a direct sum, direct product, or mixed module, depending on the shape (expanding or contracting) of
its ends. For instance, for the string $u$ shown in Figure \ref{fig1} the corresponding mixed module $M^+(u)$ is a
direct product module on the expanding end (on the right), and a direct sum module on the contracting end (on the
left).  Precisely, we take a basis element at each node of the string, form the product of the 1-dimensional
$K$-vector spaces they generate, then take the subspace of sequences which are eventually $0$ to the left; the
action of the algebra is as given by the labelling of the arrows of the string.

For each band $C$ over (any) string algebra $A$ and for a fixed $0\neq \lam\in K$ the finite dimensional band modules
$R_k= M(C,\lam, k)$ form a ray of irreducible monomorphisms, $R_1\to R_2\to \dots$, in the Auslander--Reiten quiver.
The direct limit along this ray is an indecomposable pure injective \emph{Pr\"ufer} module. Similarly, there is a
coray of irreducible epimorphisms  $R_1\xl{} R_2\xl{} \dots$ whose inverse limit is the \emph{adic} module and which
is also pure injective and indecomposable. We will refer to Pr\"ufer and adic modules, as well as the generic modules
associated to bands (one to each band, see, e.g., \cite[\S 8.1.2]{Preb2}), as \emph{infinite dimensional band modules}.

The following conjecture is due to Ringel \cite{Rin00}.

\begin{conj}\label{r-conj}
Let $M$ be an infinite dimensional indecomposable pure injective module over a domestic string algebra $A$. Then $M$
is either an infinite dimensional band module or a 1-sided or 2-sided direct sum, direct product, or mixed module
corresponding to a 1-sided almost periodic or 2-sided biperiodic string.
\end{conj}

We will denote by $M_w$ the module, described above, corresponding to the string $w$ in this conjecture, and say that
$M_w$ is on Ringel's list. For 1-domestic string algebras this conjecture was verified in \cite{Pun14}.

\section{Pure injective modules}\label{S-pi}

A module $M$ over a finite dimensional algebra $A$ is said to be \emph{pure injective} (or algebraically compact) if
it is a direct summand of a direct product of finite dimensional modules. In particular, every finite dimensional
module is pure injective. From now on $A$ will denote a (usually domestic) string algebra with a fixed presentation
by a quiver with relations.

Recall (see \cite{B-R}) that for each vertex $S$ of $Q$ one can partition the strings entering $S$, including extra
strings $1_{S,\pm 1}$, into two sets $H_{\pm 1}= H_{\pm 1}(S)$. Namely we require that $1_{S,i}\in H_i$ and each
$H_i$ contains at most one direct (and inverse) arrow ending in $S$. Furthermore for each direct arrow $\al\in H_i$
and inverse arrow $\be^{-1}\in H_i$, we insist that $\be\al$ is a relation in $A$. If $C$ is a string of length
$\geq 1$ then we put it in $H_i(S)$ if the first (that is, furthest to the left) letter of $C$ is there. Observe that this partition is often
non-unique.

For instance over $R_1$ we can choose $\be, \be^{-1}\in H_1$ and $\al, \al^{-1}\in H_{-1}$ and then take a string $u$
in $H_i(S)$ if the first letter of $u$ is there. Thus for this example $H_1(S)$ consists of $1_{S,1}$ and strings that
start with either $\be$ or $\be^{-1}$.

The strings in each $H_i(S)$ are ordered in a natural way (with $\del^{-1} < 1_{S,i} < \ga$) and each of these sets
forms a chain with respect to this ordering. For instance, with the above choice for $R_1$, we get that
$\be\al^{-1}< \be $ in $H_1$ and  $\al^{-1}< 1_{S,-1}< \al$ in $H_{-1}$. When fixing a node in a diagram such as
Figure \ref{fig1} and considering related diagrams, these two sets will distinguish between the possible strings
extending to the right and to the left from that node.

One can extend the partition $H_{\pm 1}$ to include 1-sided infinite strings: we take a 1-sided string $u$ in
$\wh H_i= \wh H_i(S)$ if the first letter of $u$ is there. The ordering on $H_i(S)$ extends naturally to a linear
ordering on the set $\wh H_i(S)$ of 1-sided (finite or infinite) strings. For instance, over $R_1$, the infinite
string $(\al\be^{-1})^{\fty}$ belongs to $\wh H_{-1}$ and is larger than each finite string $(\al\be^{-1})^n$. In
general (say, over $G_{2,3}$) the structure of this ordering is quite complicated. Over domestic algebras it is not
so complicated and the following fact follows from \cite[4.10 (and 3.2)]{Sch97}.

\begin{fact}\label{non-dense}
Let $A$ be a domestic string algebra. Then each chain $\wh H_i$ contains no subchain isomorphic to the ordering
of the rationals.
\end{fact}

Suppose that $M$ is an $A$-module, $N$ is a subspace of $M$ and $\al$ is an arrow. Then $\al N$ will denote the image
of $N$ with respect to $\al$, and $\al^{-1} N$ is the preimage of $N$. If $D\in H_1(S)$ is a string of length $k$
then define $DM$ by induction on $k$, by setting $1_{S,1} M= M$ and $DM= l (D' M)$, if $D= l D'$ for a letter $l$.

Now for each $D\in H_1(S)$ we define the pp formula $(.D)(x)$. Namely if there is an arrow $\ga$ such that $D\ga^{-1}$
is a string then this formula states that $x= e_S\, x$ and $x\in D\ga^{-1}(0)$. If no such arrow exists then we
require just $x= e_S x$ and $x\in D M$. Informally we will consider this formula as asserting divisibility (in fact a bit more)
by $D$ on the right.

For instance if we choose $\be\in H_1(S)$ over $\wt A_1$ and $D= \be\al^{-1}$ then $m\in (.D)(M)$ iff  $m\in \be M$
(in this case there is no $\ga$ as above). But over $R_1$ for the same formula we obtain $m\in (.D)(M)$ iff there
exists $n\in M$ such that $m= \be n$ and $\be\al n= 0$ (in this case $\ga$ in the definition above is $\be$).

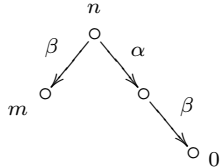
$$
\vcenter{%
\xymatrix@R=14pt@C=10pt{%
&*+={\circ}\ar[ld]_{\be}\ar[rd]^{\al}\ar@{}+<0pt,10pt>*{_n}&&\\
*+={\circ}\ar@{}+<-10pt,-6pt>*{_m}&&*+={\circ}\ar[rd]^{\be}&\\
&&&*+={\circ}\ar@{}+<8pt,-2pt>*{_0}
}}
$$

\vspace{2mm}

Observe that, if $D, D'\in H_1(S)$, then $D\leq D'$ if and only if $(.D')$ implies $(.D)$ (hence the ordering on
pp formulas inverts the ordering on strings).

The left-handed formulas $(C^{-1}.)$ for $C\in H_{-1}(S)$ are defined similarly. Finally let $(C^{-1}.D)$ denote the
conjunction of $(C^{-1}.)$ and $(.D)$.

Suppose that $M$ is an $A$-module and $m$ is a nonzero element in $e_S M$ for some primitive idempotent $e_S$. There
is standard way (see \cite{P-P04} or \cite[Sect. 5.8]{Har}) to assign to $m$ a (finite, 1-sided, or 2-sided) string
$w= u^{-1}.v$, where $u\in \wh H_{-1}(S)$ and $v\in\wh H_1(S)$.

Namely define $v(m)$, the \emph{right handed string}, as the supremum of strings $D\in H_1(S)$ such that the formula
$(.D)(x)$ holds for $m$ in $M$. For instance, if $M= M(u)$ is a finite dimensional string module, $m$ is its
leftmost basis element and $u\in H_1(S)$ then $v(m)= u$. Furthermore if $M= M^+$ is as in Figure \ref{fig1} and
$m$ is the marked element from the socle of $M$, then it is easily checked that $v(m)= \be (\al\be^{-1})^{\fty}$.

Note that if $v= v(m)$ is infinite then one can choose an ascending sequence of substrings $D_1< D_2< \dots$ of $v$
such that $v= D_k v_k$ where $D_k$ ends with an inverse arrow, $v_k$ starts with a direct arrow and $v$ is the
supremum of the $D_k$. Furthermore if $M$ is pure injective, we can always divide $m$ in $M$ by $v(m)$, meaning
that there is a sequence $m_k$ of elements in the socle of $M$ which together
witness each of the conditions $m\in (.D_k)(M)$.

The \emph{left handed string of $m$}, $u(m)$, is defined similarly. For instance in our running example we obtain
$u= (\al\be^{-1})^{\fty}$, or rather $u^{-1}= {}^{\fty} (\be\al^{-1})$ - which is the way it is shown on the diagram.

Finally we will set $w(m)= u^{-1}.v$. For instance (see \cite[L. 111]{Har}) if $M= M_w$ is from Ringel's list and
$m$ is an element from a standard basis placed between the strings $u^{-1}$ and $v$ (so $w= u^{-1}v$), then
$u= u(m)$ and $v= v(m)$. Note also that to every element in a standard basis of a finite dimensional band module we
assign in this way a 2-sided periodic string $^{\fty}D. D^{\fty}$ for some primitive cycle $D$ (we have included
among the string, rather than band, modules those `band' modules, such as that with parameter $\lambda =0$ in
Figure \ref{figband}, which can be described in both ways).

Let us make the following trivial, but important, remark that will be used frequently. Suppose that $m= \al n$
for some nonzero $m\in e_S M$, $n\in e_T M$ and $\al^{-1}\in H_{-1}(T)$. Then $u(n)= \al^{-1} u(m)$ and
$v(m)\geq \al v(n)$. Namely if $u(n)> \al^{-1} u(m)$ then there is $\ga\in H_{-1}(T)$ such that $n$ is divisible
by $\ga$, hence $\al\ga= 0$ implies $m=0$, a contradiction.

$$
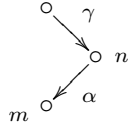
\vcenter{%
\xymatrix@C=10pt@R=10pt{%
*+={\circ}\ar[rd]^{\ga}&\\
&*+={\circ}\ar[ld]^{\al}\ar@{}+<10pt,0pt>*{_n}\\
*+={\circ}\ar@{}+<-10pt,-4pt>*{_m}&
}}
$$

\vspace{2mm}

Suppose that $M$ is a pure injective module pointed at a (nonzero) element $m\in e_S M$ and let $w(m)= u^{-1}.v$
(a \emph{pointed module} is a module with a specified element or tuple). Say that $m$ is \emph{homogeneous} if there is no
$x\in e_S M$ such that the left string of $x$ is greater than $u$ and the right string of $m-x$ is greater than $v$.
Note (by Lemma \ref{triang} below) that if an element $m$ could be written in this way then the right strings of $m$
and $x$ would be equal, as would be the left strings of $m-x$ and $m$.

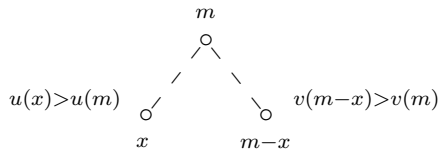
$$
\vcenter{
\xymatrix@R=20pt@C=14pt{%
&*+={\circ}\ar@{--}[dl]\ar@{--}[dr]\ar@{}+<0pt,10pt>*{_m}&\\
*+={\circ}\ar@{}+<-1pt,-10pt>*{_x}\ar@{}+<-30pt,6pt>*{_{u(x)> u(m)}}&&
*+={\circ}\ar@{}+<0pt,-10pt>*{_{m-x}}\ar@{}+<38pt,6pt>*{_{v(m-x)> v(m)}}
}
}
$$

\vspace{2mm}

For instance it follows from \cite[L. 156]{Har} that every element in the canonical basis of a module $M= M_w$ from
Ringel's list is homogeneous. Note that Harland uses `fundamental' instead of `homogeneous' (and the terminology in
\cite{BarPre} was `maximal').

We will discuss this notion in detail in the next section. The following result is crucial for our considerations.
Including the proof would make the paper considerably longer so we have given a different but direct proof of this
for the case of domestic string algebras. We have placed that at the end of the paper so as not to break the flow
towards our main result.

\begin{fact}(see \cite[Thm. 51, p. 268]{Har} and its proof)\label{har}
Suppose that $A$ is an arbitrary string algebra and let $w= u^{-1}.v$ be either a 1-sided string or a 2-sided
\textbf{non-periodic} string. Then there exists a unique indecomposable pure injective module $N_w$ containing a
homogeneous element $m$ whose string $w(m)$ equals $u^{-1}.v$.

Furthermore if $M$ is an arbitrary pure injective module with a homogeneous element $m$ with $w(m) = u^{-1}.v$, then
$N_w$ is a direct summand of $M$.
\end{fact}

Observe that over a domestic string algebra every string $w$ is almost periodic (in one or both directions), hence
$N_w$ is isomorphic to $M_w$.

\section{Homogeneous elements}\label{S-hom}

In this section we will develop the machinery of homogeneous elements, closely following Harland's thesis.

We start with a straightforward lemma, a `triangle inequality'.

\begin{lemma}\label{triang}(see \cite[L. 155]{Har})
Suppose that $m_1, m_2\in e_S M$ have right strings $v_1$ and $v_2$. Then the element $m= m_1+ m_2$ has
right string greater than or equal to $\min(v_1,v_2)$. Furthermore if $v_1\neq v_2$ then the right string
of $m$ equals $\min(v_1,v_2)$.
\end{lemma}
\begin{proof}
Suppose that $C$ is a finite string such that $C\leq v_1, v_2$, therefore $m_1$ and $m_2$ are divisible by $C$ on
the right. It follows easily by induction on the length of $C$ that $m= m_1+ m_2$ is divisible by $C$ on the right,
hence $C\leq v(m)$, and the first claim follows.

For the second claim, by symmetry we may assume that $v_1> v_2$. Suppose that $v=v(m) > \min(v_1, v_2)= v_2$ hence
$\min(v, v_1)> v_2$. Writing $m_2= m - m_1$, by the first statement of the lemma we derive a contradiction.
\end{proof}

The following proposition says that when multiplying or dividing a homogeneous element by an arrow, we can choose the
resulting element to be homogeneous.

\begin{lemma}\label{div}(see \cite[L. 158]{Har})
Suppose that $m$ is a homogeneous element in a pure injective module $M$ with $w(m)= u^{-1}.v$, where $v= lv'$ for a
direct or inverse arrow $l$. Choose an element $n\in e_T M$ such that $m= \al n$ and $v(n)= v'$ if $l=\al$ is direct,
and set $n= \be m$ if $l= \be^{-1}$ is inverse.  Assume without loss of generality that $v'$ is in $\wh H_1(T)$ and hence $l^{-1}u\in \wh H_{-1}(T)$.  Then $n$ is homogeneous, with $w(n)= u^{-1}l.v'$.

More generally, if $v(m)=Cv''$ and, working along $C$, we choose a sequence of elements as above, obtaining $m=Cn$,
then $n$ will be homogeneous and, with a suitable choice of $H_{\pm 1}$, $w(n)$ is equal to $u^{-1}C.v''$.
\end{lemma}

By symmetry all this holds also for left handed strings.

\begin{proof}
Suppose first that $l=\al$ is a direct arrow, hence $m= \al n$ and $v(n)= v'$.

As we have already mentioned $w(n)= u^{-1}\al. v'$. It remains to show that $n$ is homogeneous.

Suppose for a contradiction that $n= n_1+ n_2$ such that $u(n_1)> u(n)$ and $v(n_2)> v(n)$. Set $m_1= \al n_1$ and
$m_2= \al n_2$, so $m= m_1+ m_2$.

$$
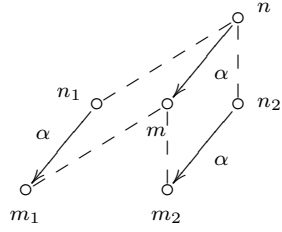
\vcenter{
\xymatrix@R=24pt@C=18pt{%
&&&*+={\circ}\ar@{}+<10pt,4pt>*{_n}\ar@{--}[dll]\ar@{--}[d]\ar[dl]^{\al}\\
&*+={\circ}\ar@{}+<-10pt,4pt>*{_{n_1}}\ar[dl]_{\al}&
*+={\circ}\ar@{}+<-4pt,-10pt>*{_m}\ar@{--}[dll]\ar@{--}[d]&
*+={\circ}\ar@{}+<12pt,0pt>*{_{n_2}}\ar[dl]^{\al}\\
*+={\circ}\ar@{}+<0pt,-10pt>*{_{m_1}}&&*+={\circ}\ar@{}+<0pt,-10pt>*{_{m_2}}&
}}
$$

\vspace{2mm}

Because $v(n_2)> v(n)$ it follows that $v(m_2)\geq \al v(n_2)> \al v(n)= v(m)$, in particular $m_1\neq 0$. Consequently
$\al^{-1} u(m_1) = u(n_1) > u(n)=\al^{-1} u$, therefore $u(m_1)> u= u(m)$. But $m= m_1 + m_2$ and $m$ is
homogeneous, a contradiction.

Now consider the case when $l= \be^{-1}$ is inverse, so $\be m= n$; note that $n\neq 0$. Then, by definition, $v(n)= v'$ and
$u(n)\geq \be u$. Suppose, by way of contradiction, that $u(n)> \be u$. Therefore we can divide $n$ by $\be$, finding
an element $m'$ such that $\be m'= n$ and $u(m')> u$.

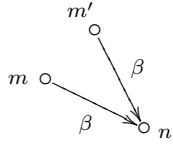
$$
\vcenter{
\xymatrix@R=10pt@C=10pt{%
&*+={\circ}\ar[rdd]^{\be}\ar@{}+<-5pt,8pt>*{_{m'}}&\\
*+={\circ}\ar[rrd]_{\be}\ar@{}+<-10pt,0pt>*{_m}&&\\
&&*+={\circ}\ar@{}+<8pt,-2pt>*{_n}
}}
$$

\vspace{2mm}

Then $\be(m-m')=0$ yields $v(m-m')> v(m)$. It follows that $m= m'+ (m-m')$ with $u(m')> u(m)$ and $v(m-m')> v(m)$, a
contradiction to homogeneity of $m$.  So $u(n)=\be u$.

It remains to show that $n$ is homogeneous. Otherwise $n= n_1+ n_2$ such that $u(n_1)> u(n)$ and $v(n_2)> v(n)$. Since
$u(n_1)> u(n)= \be u$, $n_1$ is divisible by $\be$: there exists $m_1$ such that $\be m_1= n_1$ and $u(m_1)> u$.

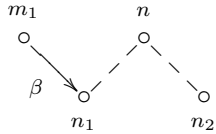
$$
\vcenter{
\xymatrix@R=14pt@C=14pt{%
*+={\circ}\ar@{}+<0pt,10pt>*{_{m_1}}\ar[rd]_{\be}&&
*+={\circ}\ar@{}+<0pt,10pt>*{_n}\ar@{--}[ld]\ar@{--}[rd]&\\
&*+={\circ}\ar@{}+<0pt,-10pt>*{_{n_1}}&&*+={\circ}\ar@{}+<0pt,-10pt>*{_{n_2}}
}}
$$

\vspace{2mm}

Because $m$ is homogeneous, from the decomposition $m= m_1 + (m-m_1)$ it follows that $v(m-m_1)\leq v= \be^{-1}v'$,
therefore $v(\be (m-m_1))\leq v'$. But $\be(m-m_1)= n_2$ and $v(n_2)> v(n)= v'$, a contradiction.

The statement for general finite strings $C$ follows by induction.
\end{proof}

We say that a pure injective module $M$ is \emph{1-sided}, if it contains a nonzero element $m\in e_S M$ whose
string $w(m)$ is 1-sided (or finite), and $M$ is \emph{2-sided} otherwise. For instance it can be checked that the module $M_w$
corresponding to a string $w$ on Ringel's list is 1-sided if and only if $w$ is 1-sided. Furthermore each finite
dimensional band module is 2-sided.

One-sided indecomposable pure injective modules over (any!) string algebra $A$ were classified in \cite{P-P04}: they
one-to-one correspond to the 1-sided strings. In particular, if $A$ is a domestic string algebra then each 1-sided
module is on Ringel's list. Thus in classifying indecomposable pure injective modules it suffices to consider the
2-sided case. Furthermore, by Fact \ref{har}, if $M$ contains a homogeneous element whose string is non-periodic,
then we know the structure of $M$, in particular it is on Ringel's list.

From the next proposition it follows that, in the domestic case, every pure injective module contains a homogeneous
element. That will leave us with the periodic 2-sided case to deal with in the next section.

\begin{prop}\label{dense}(see \cite[L. 168]{Har})
Let $A$ be an arbitrary string algebra and let $M$ be a 2-sided pure injective $A$-module containing no homogeneous
elements. Take any $0\neq m\in e_S M$ and set $w(m)= u^{-1}.v$. Then for any string $u'> u$ there exists $x\in e_S M$
such that $u'> u(x)> u$ and $v(m-x)> v$.
\end{prop}

It follows from Lemma \ref{triang} that $v(x)= v$ and $u(m-x)= u$ in this case.

\begin{proof}
Let $X$ be the set of finite strings $D\in H_1$, $D> v$ such that $m$ can be written as a sum $m_1+ m_2$, where $m_1$
is divisible by $E$ for some finite string $E\in H_{-1}$ greater than $u$, and $m_2$ is divisible by $D$. Because $m$
is not homogeneous this set is not empty. Because $M$ is 2-sided it follows immediately that the supremum $v'$ of $X$
is an infinite (1-sided) string larger than $v$.

Because $v'$ is infinite it can be written as $D_i l_i t_i$ for some direct arrows $l_i$ and finite strings $D_i\in X$
of strictly increasing length, each ending with an inverse arrow; in particular $D_1< D_2< \dots < v'$ in
$\wh H_1$ and $D_i$ is a \emph{presubstring} of $v'$ in the terminology of \cite{Har} (described as ``closed under
predecessors" in \cite{CB}).

For each positive $k\in \N$ consider the set of strings $F> u$ such that $m$ can be written as a sum $m= m_1+ m_2$,
where $m_1$ is divisible by $F$ and $m_2$ is divisible by $D_k$. Because $D_k\in X$, this set is non-empty; let $u(k)$
denote the supremum of strings in this set. Then $u(k)> u$ and clearly $u(k)\geq u(k+1)$ for each $k$. We claim that
the infinum of the $u(k)$ equals $u$, from which the result follows.

Suppose for a contradiction that this is not the case, so choose a finite string $C$ with $u< C< \inf_k u(k)$. Consider
the (infinite in $k$) set of (pp) conditions in a variable $x$ saying that $m$ can be written as a sum $x+ (m-x)$
such that $x$ is divisible by $C$ and $m-x$ is divisible by $D_k$. By assumption and construction, any finite subset
of this set is satisfied by some value of $x$ in $M$.

Since $M$ is pure injective ($=$ algebraically compact) there is a simultaneous solution for this set of conditions.
Namely there exists $x\in M$ such that $x$ is divisible by $C$ and $m-x$ is divisible by $D_k$ for each $k$. It
follows that the right string of $m-x$ is at least $v'= \sup X$, and therefore equals $v'$.

Because $m-x$ is not homogeneous, there exists $y$ such that $u(y)> u(m-x)= u$ and $v(m-x-y)> v(m-x)= v'$.

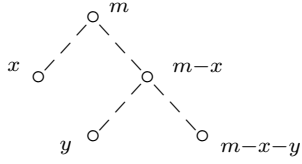
$$
\vcenter{
\xymatrix@R=14pt@C=12pt{%
&*+={\circ}\ar@{--}[ld]\ar@{--}[rd]\ar@{}+<10pt,4pt>*{_m}&&\\
*+={\circ}\ar@{}+<-9pt,4pt>*{_x}&&*+={\circ}\ar@{--}[ld]\ar@{--}[rd]\ar@{}+<19pt,4pt>*{_{m-x}}&\\
&*+={\circ}\ar@{}+<-10pt,-4pt>*{_y}&&*+={\circ}\ar@{}+<22pt,-4pt>*{_{m-x-y}}
}}
$$

\vspace{2mm}

Now, $m= x+ y+ (m-x-y)$. However $x+y$ has left string greater than $u$ (because $x$ and $y$ do). Thus there is a
finite string $C'> u$ such that $x+y$ is divisible by $C'$. Furthermore the right string of $m-x-y$ is greater than
$v'= \sup X$, so there is a finite string $D> v'$ such that this element is divisible by $D$. It follows that
$D\in X$, a contradiction.
\end{proof}

The following is derived immediately.

\begin{cor}\label{hom-dom}
Let $M$ be any nonzero pure injective module over a domestic string algebra. Then $M$ contains a homogeneous element.
\end{cor}
\begin{proof}
This follows from Proposition \ref{dense} and Fact \ref{non-dense}.
\end{proof}

\section{Main result}\label{S-main}

Before proving the main result we have to introduce the \emph{bridge quiver} of a domestic string algebra $A$. It will
be more convenient for this paper to define it as a poset, rather than as a directed graph (see \cite{Sch97}) (though
this might lose information - for instance there can be more than one path between points in the bridge quiver).

We fix a set $\mB$ of representatives of bands over $A$ up to a cyclic permutation (but a band and its inverse are
represented by different elements in $\mB$). For instance, for $\wt A_1$ or $R_1$, this set consists of two elements
$C= \al\be^{-1}$ and $C^{-1}= \be\al^{-1}$.

We define $C\preccurlyeq D$ for bands $C, D\in \mB$ if there is a finite string of the form $CuD$ over $A$. It
follows from the combinatorics of domestic string algebras (for example, from Fact \ref{band} (ii), (iii) below) that
$C\preccurlyeq D\preccurlyeq C$ implies $C= D$, hence $\mB$ is a poset. For instance, for $A= \Lam_2$ we can choose
$C= \eps\del^{-1}$, $D= \al\be^{-1}$ together with $C^{-1}, D^{-1}$ as the elements of $\mB$. Then the bridge quiver
of $\Lam_2$ is a union of two chains, $C\prec D$ and $D^{-1} \prec C^{-1}$, the comparison $C\prec D$ being realized
by the string $C\eps\ga D$.

It follows from \cite{Sch97} that for domestic string algebras $\mB$ is a finite poset.

Now we are in a position to prove the main result of the paper.

\begin{theorem}\label{main}
Every indecomposable pure injective module over a domestic string algebra $A$ is on Ringel's list.
\end{theorem}

The proof of this result will need a few lemmas. Let $M$ be an indecomposable pure injective $A$-module. As we already
mentioned, if $M$ is 1-sided, then it is on Ringel's list. Thus we may assume that $M$ is 2-sided. Furthermore by
Corollary \ref{hom-dom} we know that $M$ contains a homogeneous element. If the string of this element is non-periodic
then, by Fact \ref{har}, $M$ is on Ringel's list.

Thus we may assume for the rest of the proof that $M$ is \textbf{2-sided, contains a homogeneous element, and any
homogeneous element in $M$ is periodic} (i.e.\ its string is 2-sided and periodic).

We recall the following facts about the combinatorics of bands over domestic string algebras.

\begin{fact}(see \cite{Rin95})\label{band}
Suppose that $E, F$ are bands over a domestic string algebra.

(i) If $\al^{-1}\be$ occurs in both $E$ and $F$ then these bands are equivalent up to cyclic permutation.

(ii) If both $E$ and $F$ begin with $\al$ then $E=F$.

(iii) If $E$ begins with $\al$, ends with $\be^{-1}$ and the same is true for a string $G$ then $G$ is a power of $E$.
\end{fact}

Let us say that a band $C= \ga\dots \del^{-1}$ is \emph{realized} in $M$ if there is a nonzero element $n\in \ga M\cap \del M$ (we do not assume $n$ to be homogeneous), in particular $n$ belongs to the socle of $M$.

Choose $D= \al\dots \be^{-1}$ to be any band which is minimal in the ordering on the bridge quiver $\mB$ among all
bands realized in $M$; we may assume that $\al\in H_1$ and $\be\in H_{-1}$.

\begin{lemma}\label{D}
Let $D$ be as above.  Suppose that $ m$ is any nonzero element in $ \be M$. Then $m$ is homogeneous with
$w(m)= {}^{\fty} D.D^{\fty}$.
\end{lemma}
\begin{proof}
First we will check that $u(m)^{-1}= {}^{\fty} D$. Now, $u(m)^{-1}$ is an infinite string which ends with $\be^{-1}$;
since $A$ is domestic, this string is almost periodic: $u(m)^{-1}= {}^{\fty} C u' \be^{-1}$ for some band $C$ and
some finite string $u'$. If $C$ is a cyclic permutation of $D$ then, changing $u'$ if necessary, we may assume that
$C=D$, hence $u(m)^{-1}= {}^{\fty} D$ by Fact \ref{band} (iii).

Otherwise $C$ is not a cyclic permutation of $D$. Because $u(m)^{-1} D$ is, clearly, a string it follows that
$C\prec D$ in the bridge quiver. Since $C$ is obviously realized in $M$, we get a contradiction to minimality of $D$.

Second, we show that $m$ is homogeneous. If not, then $m= x+ (m-x)$ where $u(x)> u(m)$ in $\wh H_{-1}$ and
$v(m-x)> v(m)$ in $\wh H_1$, in particular $x\neq 0$. From $u(x)\geq  u(m)$ it follows that $x\in \be M$, hence
$u(x)^{-1}= {}^{\fty} D$ by the first part of the proof. But then $u(x)= u(m)$, a contradiction.

Finally, $v(m)= D^{\fty}$ (in particular $m$ is divisible by $\al$) since we are assuming that homogeneous elements
have periodic words.
\end{proof}

For the remaining portion of the proof we will assume some familiarity with free realizations of pp formulas
(see \cite[\S 1.2.2]{Preb2}). Namely (in our context), we say that a pointed finite dimensional module $(N,n)$ is a
\emph{free realization} of a pp formula $\phi(x)$ if $n$ satisfies $\phi$ in $N$ and, for any element $l\in L$
satisfying $\phi$ in a module $L$, there exists a morphism $f: N\to L$ sending $n$ to $l$.

Every formula has a finite dimensional free realization (see \cite[1.2.14]{Preb2}). Recall that if $a\in A$ then the
\emph{divisibility formula} $a\mid x$ claims that there exists $y$ such that $x= ay$, hence it defines in each left
$A$-module $M$ the subspace $aM$. For instance this divisibility formula has $(A,a)$ as a free realization; furthermore
if $a\in Ae$ for some idempotent $e$, then the pointed module $(Ae,a)$ is another free realization of $a\mid x$.

For another example, the string module $M(\be\al^{-1}\be^{-1})$ pointed at the left end is a free realization of the
divisibility formula $\be\mid x$ over $R_1$ which therefore is equivalent to the formula $(.\be\al^{-1}\be^{-1})$.

$$
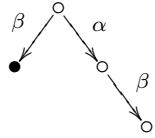
\vcenter{
\xymatrix@R=14pt@C=8pt{%
&*+={\circ}\ar[dl]_{\be}\ar[dr]^{\al}&&\\
*+={\bullet}&&*+={\circ}\ar[dr]^{\be}&\\
&&&*+={\circ}
}}
$$

\vspace{2mm}

Observe that over $\wt A_1$ the same formula $\be\mid x$ has $M(\be\al^{-1})$ as a free realization, hence is
equivalent to the formula $(.\be\al^{-1})$.

$$
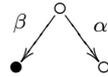
\vcenter{
\xymatrix@R=14pt@C=8pt{%
&*+={\circ}\ar[dl]_{\be}\ar[dr]^{\al}&&\\
*+={\bullet}&&*+={\circ}&\\
}}
$$

\vspace{2mm}

The following lemma says that $D$ covers the support of $M$.

\begin{lemma}\label{ga}
If $\ga$ is an arrow not occurring in $D$ or $D^{-1}$ then $\ga M= 0$.
\end{lemma}

After having proved this lemma, we will be able to assume that every arrow occurs in $D$ or $D^{-1}$.

\begin{proof}
Otherwise choose $0\neq n\in \ga M$. Take $m$ to be a homogeneous element of $M$ whose string is\
${}^{\fty} D. D^{\fty}$. Since $M$ is pure injective and indecomposable, by \cite[4.3.72]{Preb2} there is a pp formula
$\phi(x,y)$ such that $M\ms \phi(m,n)\wg \neg \phi(m,0)$ (``$M\ms \phi(m,n)\wg \neg \phi(m,0)$" may be  read as
``the pair $(m,n)$ satisfies the condition $\phi$ in $M$ but $(m,0)$ does not"; the fact that $M$ is indecomposable
pure injective means that there is a non-trivial link, expressed by a pp formula, between any nonzero elements).
We will show that this leads to a contradiction.

We may assume that $\phi(x,y)$ implies $(D.D)(x)$ and implies $\ga\mid y$ (both these are pp conditions, so may be
incorporated into $\phi$). We may further suppose, possibly changing $m$, that $\phi$ has an indecomposable free
realization. To see this, decompose a free realization, $N$, of $\phi$ as a direct sum of indecomposables,
$N=\bigoplus N_i$, and write $\phi$ as a corresponding sum of formulas $\phi_i$ with indecomposable free realizations;
then decompose $(m,n)= (\sum m_i, \sum n_i)$ accordingly, so $(m_i, n_i)$ realizes $\phi_i$.

Since $\phi$ implies $ (D.D)(x)$ so does each $\phi_i$, and similarly $\phi_i$ implies $\ga\mid y$. It follows from
Lemma \ref{D} that either $m_i= 0$ or $m_i$ is a homogeneous element with string $^{\fty} D. D^{\fty}$. Furthermore
$n_i\in \ga M$.

For each $j$ we have $M\ms \phi(m_j,n_j)$ and there is $i$ such that $M\ms \neg \phi_i(m_i,0)$ (in particular both
$m_i, n_i$ are nonzero).  Otherwise, for each $i$ we would have $M\ms \phi_i(m_i, 0)$ and, adding up, would get
$M\ms \phi(m,0)$, a contradiction. Now replace $\phi$ with $\phi_i$, $m$ with $m_i$ and $n$ with $n_i$.

Therefore we can choose $m$, $n$ and $\phi$ so that $\phi$ has a finite dimensional indecomposable free realization,
say $(c,d)\in N$.  There is a morphism $f: N\to M$ such that $f(c)= m$ and $f(d)= n$.

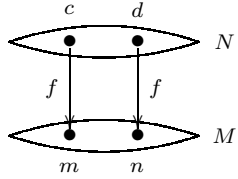
$$
\vcenter{%
\xymatrix@C=20pt@R=30pt{%
*+={}\ar@{-}@/^.5pc/[rrr]\ar@{-}@/_.5pc/[rrr]&*+={\bullet}\ar[d]_f\ar@{}+<0pt,12pt>*{_c}&
*+={\bullet}\ar[d]^f\ar@{}+<0pt,12pt>*{_d}&*+={}\ar@{}+<10pt,0pt>*{_N}\\
*+={}\ar@{-}@/^.5pc/[rrr]\ar@{-}@/_.5pc/[rrr]&*+={\bullet}\ar@{}+<0pt,-12pt>*{_m}&
*+={\bullet}\ar@{}+<0pt,-12pt>*{_n}&*+={}\ar@{}+<10pt,0pt>*{_M}
}}
$$

\vspace{2mm}

Suppose first that $N$ is a string module, $N=M(G)$ say, where $G$ is a finite string. Choose a standard basis of $N$
and write $c= \sum_i \lam_i c_i$ with each $c_i$ in this basis and each $\lam_i\neq 0$. By \cite[4.2]{P-P04}, $c_i$
is divisible by $D$ to both the right and the left, in particular $c_i\in \al N\cap \be N$ so, if we set $m_i=f(c_i)$
then, by Lemma \ref{D}, $m_i=0$ or $w(m_i)= {}^\infty D.D^\fty$.

The reader could have in mind the following diagram for $A= R_1$, where $D= \be\al^{-1}$ and $N$ is the following
string module:

$$
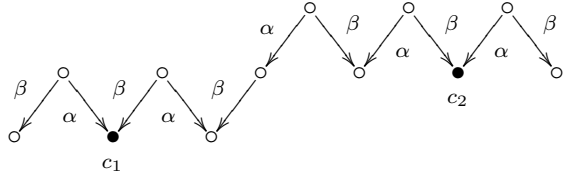
\vcenter{%
\xymatrix@C=10pt@R=16pt{%
&&&&&&*+={\circ}\ar[dl]_{\al}\ar[dr]^{\be}&&*+={\circ}\ar[dl]^{\al}\ar[dr]^{\be}&&
*+={\circ}\ar[dl]^{\al}\ar[dr]^{\be}&\\
&*+={\circ}\ar[dl]_{\be}\ar[dr]_{\al}&&*+={\circ}\ar[dl]_{\be}\ar[dr]_{\al}&&
*+={\circ}\ar[dl]_{\be}&&*+={\circ}&&*+={\bullet}\ar@{}+<0pt,-10pt>*{_{c_2}}&&*+={\circ}\\
*+={\circ}&&*+={\bullet}\ar@{}+<0pt,-10pt>*{_{c_1}}&&*+={\circ}&&&&&&&
}}
$$

\vspace{3mm}

However this diagram is misleading. For instance $v(c_1)= \be \al^{-1}\be\al\dots$ is larger than
$D^{\fty}= \be\al^{-1}\be\al^{-1}\dots$ which cannot happen (see arguments below).

Write $G=G_i.H_i$ where $c_i$ lies at the cut point and set $H_i=H_i'H_i''$ where the length of $H_i'$ is equal to
the length of $D$; then we must have $H_i'=D$: for $H_i$ must be at least $D$ since $c_i$ is divisible by $D$ to
the right, but it cannot be larger than $D$ since then the right word of $m_i$, the image of $c_i$, would begin with
a word larger than $D$ (unless $m_i= 0$).  It could be that $c_i$ is so far to the right in $G$ that the length
of $H_i$ is less than that of $D$, in which case, by the same reasoning, $H_i$ must be an initial substring of $D$
(and will be a presubstring of $D$).

Similarly write $d= \sum_j \mu_j d_j$ such that each $d_j$ is an element of the chosen basis and each $\mu_j\neq 0$.
Again by choice of $\phi$ and \cite[4.2]{P-P04}, $d_j \in \ga N$, in particular $c_i\neq d_j$ for all $i, j$. We will
show that, in fact, the $c_i$ and $d_j$ are sufficiently separated in $G$. Namely we can find a submodule $L$ of $N$
which contains all the $d_j$ but no $c_i$ (with $m_i\neq 0$) and such that, if we factor out $L$,  by
$\pi_L: N\la N'= N/L$, we can then embed $N'$ into $M$ so that the image of $c$ under the composite map is exactly
$m=f(c)$.

$$
\vcenter{%
\xymatrix@C=20pt@R=18pt{%
*+{N}\ar[r]^{\pi_L}\ar[d]_{f}&*+{N'}\ar[dl]^{f'}\\
*+{M}&
}}
$$

\vspace{2mm}

In order to embed $N'$ in $M$ we have to ensure that each segment of the string $G$ which remains after factoring out
$L$ is a presubstring of $^\fty D.D^\fty$.  We will then be able to complete the argument since, from $N\ms \phi(c,d)$
we will obtain $N'\ms \phi(\pi_L(c),0)$ and hence $M\ms \phi(m,0)$, which is a contradiction.

We must define $L$.  Fix $d_j$; we will find a presubstring of $G$ which contains $d_j$ but no $c_i$, with $m_i\neq 0$
(terminologically, we will confuse strings and their realizations). Suppose that there is some $c_k$ to the left of
$d_j$; choose the nearest one, say $c_i$. Consider the portion of $G$ between $c_i$ and $d_j$; as seen above, it has
the form $DH$ for some (at first sight possibly empty) string $H$.

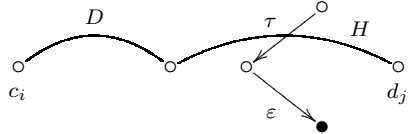
$$
\vcenter{%
\xymatrix@C=20pt@R=14pt{%
&&&&*+={\circ}\ar[dl]_{\tau}&\\
*+={\circ}\ar@{}+<0pt,-10pt>*{_{c_i}}\ar@{-}@/^1pc/[rr]^{D}&&*+={\circ}\ar@{-}@/^1pc/[rrr]^(.8){H}&
*+={\circ}\ar[dr]_{\varepsilon}&&*+={\circ}\ar@{}+<0pt,-10pt>*{_{d_j}}\\
&&&&*+={\bullet}&
}}
$$

\vspace{2mm}

Note that either $H$ ends with $\ga^{-1}$ or $DH$ continues in $G$ as $DH\ga$.  Working along $DH$ from $c_i$,
consider the first letter where $DH$ (or $DH\ga$) differs from, and hence, as argued above, is strictly less than,
$D^\fty$; say $D^\fty$ has $\tau$ at that point and $DH$ has $\eps^{-1}$.   Clearly this alternation of letters
occurs strictly on the left of $d_j$ (otherwise $\ga$ will be a part of $D$, a contradiction). We put the (marked)
image of $\eps$ into $L$ along with all subsequent basis elements up to and including $d_j$.

On the other hand, since the corresponding letter, $\tau$, in $D^\fty$ is direct, the remaining string will, at that
point, be closed under predecessors in $D^\infty$, allowing us to use the restriction of $f$ to define the embedding
of $N'$ into $M$.  We use the same procedure to the right of $d_j$ and note that what we have put into $L$ is indeed
given by a presubstring, hence is a submodule of $N$.

Now assume that $N= M(E, \lam, k)$ is a band module for some band $E$. If $E$ is equivalent to $D$ then, since $\ga$
does not appear in $D$, $\ga N= 0$, hence $d=0$. Applying $f$ we obtain $M\ms \phi(m,0)$, the required contradiction.

Suppose, then, that $E$ is not equivalent to $D$. Since $\phi$ implies $ (D.D)(x)$, there is a morphism
$g: M(D.D)\to N$ sending the basis element $z\in M(D.D)$, between two copies of $D$, to $c$.  The description of
morphisms between string and band modules (see \cite[Sec. 6.3.2]{Har}) says that $g$ is a linear combination
$\sum_i \mu_i g_i$ of simple string maps (also called graph maps in \cite{Kra91}). Consider the maps $g_i$.

By the description of simple string maps, $g_i$ is given by first possibly removing arrows at either or both ends of
$D.D$ so that it equals a presubstring of\, ${}^{\fty} E^{\fty}$ - as in the argument above, this corresponds to mapping
$M(D.D)$ to a factor module and then embedding that into the direct sum module $M({}^{\fty} E^{\fty})$ - and then
applying a canonical morphism from $M(^{\fty} E^{\fty})$ to the band module $N$. Because, by Fact \ref{band}(ii) $D$
is not a presubstring of\, $^{\fty} E^{\fty}$, it follows that proper factoring must occur at each end of $D$. When
factoring we will properly increase the left and the right string of $z$ making each definitely larger than $D^{\fty}$.

Now, applying $g$ to $z$ then $f$ to $c$, we see that $m$ is a sum of elements $m_i$ such that the left and right
string of each $m_i$ exceeds $D^{\fty}$. Since $m$ is homogeneous this is not possible.
\end{proof}

Replace $A$ by the string algebra obtained by factoring out the ideal generated by all arrows that do not occur in $D$.
By the result we have just proved, this ideal is contained in the annihilator of $M$, so we may assume that each arrow
in the quiver defining $A$ occurs in $D$ or $D^{-1}$.  From this it now follows that $A$ is 1-domestic.

\begin{lemma}\label{1-dom}
Every band $E$ of $A$ is a cyclic permutation of $D$ or $D^{-1}$, therefore $A$ is 1-domestic.
\end{lemma}
\begin{proof}
Suppose that $E= \eps\dots \pi^{-1}$. By assumption $D$ contains $\eps$ or $\eps^{-1}$, and also $D$ contains $\pi$ or
$\pi^{-1}$. It follows from \cite[Prop. 5.3]{Pun07} that either $D$ or $D^{-1}$ is a cyclic permutation of some band
$F= \eps \dots \pi^{-1}$. Since $A$ is domestic, we conclude that $F= E$.
\end{proof}

Thus we may assume that $A$ is 1-domestic and therefore, in view of \cite[Thm. 9.1]{Pun14}, $M$ is on Ringel's list.
This completes the proof of Theorem \ref{main}.

Recall that a (nonzero) module $M$ is said to be \emph{superdecomposable} if $M$ contains no (nonzero) indecomposable
direct summand. It is known that many (conjecturally all - see \cite{Pun04} for a precise statement) non-domestic
string algebras posses a superdecomposable pure injective module. This never happens for domestic string algebras, as
the following result shows.

\begin{theorem}\label{sup}
Every pure injective module over a domestic string algebra contains an indecomposable direct summand.
\end{theorem}
\begin{proof}
Let $M$ be a pure injective $A$-module. By Corollary \ref{hom-dom} $M$ contains a homogeneous element $m$. If the
string $w(m)$ is not periodic, it follows from Fact \ref{har} that the indecomposable pure injective module $M_w$
from Ringel's list is a direct summand of $M$.

Thus we may assume that each homogeneous element in $M$ has a periodic string. As in the proof of Theorem \ref{main}
choose a homogeneous element $m\in M$ whose string $w(m)= {}^{\fty} D. D^{\fty}$ is minimal. Let $M'$ be the hull of
$m$ in $M$: this is a direct summand of $M$ which contains $m$ and is minimal such (see, e.g., \cite[\S 4.3.5]{Preb2});
it has the property (see \cite[4.3.74]{Preb2}) that every nonzero element of $M'$ is related to $m$ by a pp formula,
as at the start of the proof of Lemma \ref{ga} (this is the only point in that proof where we used indecomposability
of $M$). Then, just as above, we obtain that $M'$ is a module over a 1-domestic algebra.  But, by
\cite[Cor. 6.7]{Pun12}, there is no superdecomposable pure injective module over any 1-domestic string algebra so
$M'$, and hence $M$, has an indecomposable direct summand.
\end{proof}

\section{Reproving Harland's theorem for domestic string algebras}\label{S-repr}

In this section we develop a different approach to the proof of Fact \ref{har} that we need from Harland.

\begin{theorem} (see \cite[Thm. 51]{Har})\label{unique}
Suppose that $A$ is a domestic string algebra.  Let $w= u^{-1}.v$ be a 2-sided non-periodic string. Then $M_w$
is the unique indecomposable pure injective module which contains a homogeneous element $m$ with string
$w(m)= u(m)^{-1}.v(m)$ equal to $w$.
\end{theorem}

The proof given in \cite{Har} works for arbitrary string algebras. Our approach will give the result just for domestic
string algebras.

\begin{proof}
Suppose that $M$ is an indecomposable pure injective module containing a homogeneous element $m$ with $w(m)=u^{-1}.v$
and let $p$ be the pp type of $m$ in $M$.  Write $p^-$ for the set of pp formulae not in $p$ and, for emphasis, $p^+$ for the set of those which are in $p$. Using Lemma \ref{div}, if necessary, we may assume that $m$ is in the socle
of $M$.

We will choose a pair of pp formulae $\psi< \phi$ with $\phi\in p$, $\psi\notin p$ and show that the restriction
of $p$ to the interval $[\psi,\phi]$ is independent of $M$ and $m$. By Ziegler \cite[L. 7.10]{Zie} it will follow that the
isomorphism type of $M$ is uniquely determined, and hence $M$ is isomorphic to $M_w$.

Since $A$ is domestic it follows that $u^{-1}= {}^{\fty} E u'$ and $v= v' F^{\fty}$, where $E= \al \dots \be^{-1}$
and $F= \eps\dots \pi^{-1}$ are bands which are not equal up to cyclic permutation.  We
set $C^{-1}= E u'$, $D= v'F$, so $\phi= (C^{-1}.D)\in p$. Also write $C^{-1}= \al E'$
and $D= F'\pi^{-1}$ and consider the formula $\psi= (E'.D)+ (C^{-1}.F')\in p^-$. Since $m$ is in the socle of $M$,
it follows that $C$ and $D$ begin with direct arrows.

By the construction of $u(m)$ and $v(m)$, the element $m$ will satisfy each formula $(C_1^{-1}.D_1)$, where
$C_1\in H_{-1}$ is such that $C_1\leq u$ and $D_1\in H_1$ is such that $D_1\leq v$. Furthermore, the set of such formulae
in $p$ is clearly closed with respect to finite conjunctions.

Furthermore, because $m\in M$ is homogeneous, $m$ satisfies no formula $(E_2^{-1}.D)+ (C^{-1}.F_2)$ for $E_2>u$ in $H_{-1}$
and $F_2>v$ in $H_1$, and the set of such formulae in $p^-$ is closed with respect to finite sums.

Each formula in the interval $[\psi, \phi]$, except $\phi$, is obtained as follows: take any formula $\chi$ strictly
below $\phi$ and add it to $\psi$. We will show that exactly one of the following holds.

\vspace{1mm}

\noindent 1) $\chi$ is implied by a formula $(C_1^{-1}.D_1)\in p$ as above, and hence
$\chi+ \psi$ must be in $p^+$;

or

\noindent 2) $\chi$ implies a formula $(E_2^{-1}.D)+ (C^{-1}.F_2)\in p^-$, therefore $\chi+ \psi$ implies a formula of similar shape and hence must be in $p^-$.

\vspace{1mm}

Note that this is independent of $M$.

Fix $\phi$ and $\psi$ as above. Denote by $n\in M(C^{-1}.D)$ the element in a standard basis of $M(C^{-1}.D)$
between $C^{-1}$ and $D$; so $n\in M(C^{-1}.D)$ is a free realization of $\phi$.

Choosing a formula $\chi$ strictly below $\phi$ is equivalent to picking a morphism $f$ from $(M(C^{-1}.D), n)$ to a
finite dimensional pointed module $(L,l)$ (such that $l\in L$ is a free realization of $\chi$) such that $f$ is not
a split embedding. Because the sets of formulae in 2) are closed with respect to finite sums, we may assume that $L$ is indecomposable, hence either a string or a band module.

If $L$ is a band module then, from the description of morphisms between string and band modules, it follows, as in
the proof of Lemma \ref{ga}, that, in applying $f$, the string $C^{-1}D$ is shortened at one or both ends and hence $l\in L$
satisfies $\psi$. Then $\chi$ implies $\psi$, hence we are in case 2) and $\chi \in p^-$.

Therefore we may assume that $L$ is a string module.  Write $l$ as a linear combination of basis elements $l_i$ in
$L$, each $l_i$ corresponding to a simple string map from $M(C^{-1}.D)$ to $L$. Since $C$ and $D$ start with direct
arrows each $l_i$ lies in the socle of $L$.

As in the case that $L$ is a band module, any $l_i$ arising from a simple string map which first involves a proper
factorization of $C^{-1}.D$ on one (or both) ends is a free realization of a pp formula below $\psi$, so we can
ignore these elements for the purpose of deciding whether or not $\chi \in p$. Thus we may assume that the $H_{-1}$
string, $C_i$, of each $l_i$ in $L$ is an extension of $C$ (i.e.~$C$ is a presubstring of $C_i$); and the $H_1$
string, $D_i$, of each $l_i$ in $L$ is an extension of $D$.

If $l_i$ and $l_j$ (or rather their strings) for $i\neq j$ are embedded in $L$ with the same orientation (say, with $H_{-1}$ to
the left) then we obtain a contradiction as follows.  We can suppose, without loss of generality, that $l_i$ lies
`to the left' of $l_j$:

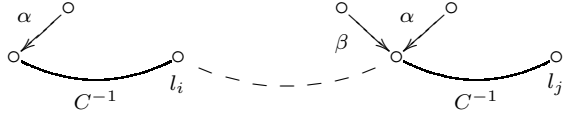
$$
\vcenter{%
\xymatrix@C=12pt@R=10pt{%
&*+={\circ}\ar[dl]_{\al}&&&&&*+={\circ}\ar[dr]_(.3){\be}&&*+={\circ}\ar[dl]_{\al}&\\
*+={\circ}\ar@{-}@/_.7pc/[rrr]_{C^{-1}}&&&*+={\circ}\ar@{}+<0pt,-10pt>*{_{l_i}}\ar@{--}@/_.9pc/[rrrr]&&&&
*+={\circ}\ar@{-}@/_.7pc/[rrr]_{C^{-1}}&&&
*+={\circ}\ar@{}+<0pt,-10pt>*{_{l_j}}&
}}
$$

\vspace{2mm}

The string between two occurrences of $\al$ on this diagram, by Lemma \ref{band} (iii), is a power of $E$. But that is
impossible by the non-periodicity of $C^{-1}.D$.

If there is just one $l_i$ (that is, $l$ is a standard basis element), then the formula $\chi$ is equivalent to a
formula of the form $(G^{-1}.H)$ generating the pp type of $l_i$ in $L$. If $G\leq u$ and $H\leq v$, then this formula
is in $p^+$, hence we are in case 1) and $\chi \in p^+$. Otherwise, say $G> u$, therefore $(G^{-1}.H)$, hence $\chi$, already is in
$p^-$.

Thus we have reduced to the case that $l= l_1+ l_2$, with $l_1$ embedded in $L$ as $C_1^{-1}.D_1$ (with left to right
orientation), and $l_2$ embedded somewhere to the right of $l_1$ with inverse orientation $D_2^{-1}.C_2$ of strings
(the case with $l_1$ to the right of $l_2$ is treated just like this one).

Because $C$ is a presubstring of $C_1$ and $C_2$, we have $C\leq C_1, C_2$ in $H_{-1}$ and similarly $D\leq D_1, D_2$
in $H_1$.

We claim that we can assume that $C_1\leq u$ and $D_1\leq v$ (or a similar assertion for $C_2$ and $D_2$). Otherwise,
by symmetry we may suppose that $D_1> v$. If $D_2> v$ then, by Lemma \ref{triang}, $l$ is divisible by
$\min(D_1,D_2)> v$ and hence (since $l$ freely realizes $\chi$ in $L$) $\chi$ implies divisibility by $\min(D_1,D_2)$,
therefore we are in case 2) and $\chi\in p^-$.

Thus we may assume that $D_2\leq v< D_1$. We may further suppose that the strings $C_1^{-1}.D_1$ and $C_2^{-1}.D_2$
are incomparable. For otherwise the pp type of $l_2$ in $L$ is strictly less than the pp type of $l_1$, therefore we
can get rid of $l_1$ without affecting $\chi$ (the pp types of $l_2$ and $l_1+l_2$ will be equal).

It follows that $C_1< C_2$. If $C_2> u$ then we have that $(C_1.D_1)+(C_2.D_2) \in p^-$ so $\chi\in p^-$.

Otherwise $C_1< C_2\leq u$, hence we have obtained the desired conclusion for $C_2$ and $D_2$, proving the claim.

Thus we may assume that $C\leq C_1\leq u$ and $D\leq D_1\leq v$. Now, $D$ is embedded as the start of $D_1$ on the
right of $l_1$ and $D^{-1}$ is embedded as the start of $D_2^{-1}$ on the left of $l_2$; these two copies of $D$
cannot overlap each other (or even touch), otherwise we would obtain a configuration $\tau \tau^{-1}$ for some
arrow $\tau$ or its inverse, which is not possible.

Thus we have obtained the following configuration in $L$:

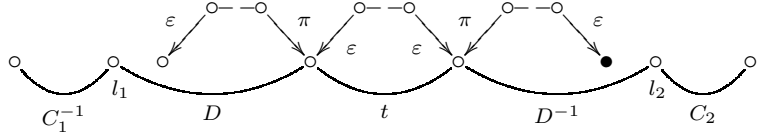
$$
\vcenter{%
\xymatrix@R=12pt@C=10pt{%
&&&&*+={\circ}\ar[dl]_{\eps}\ar@{--}[r]&*+={\circ}\ar[dr]^{\pi}&&*+={\circ}\ar[dl]^{\eps}\ar@{--}[r]&
*+={\circ}\ar[dr]_{\eps}&&*+={\circ}\ar[dl]_{\pi}\ar@{--}[r]&*+={\circ}\ar[dr]^{\eps}&&\\
*+={\circ}\ar@{-}@/_1pc/_{C_1^{-1}}[rr]&&*+={\circ}\ar@{}+<3pt,-10pt>*{_{l_1}}\ar@{-}@/_1pc/_{D}[rrrr]&
*+={\circ}&&&*+={\circ}\ar@{-}@/_1pc/_{t}[rrr]&&&*+={\circ}&&&
*+={\bullet}&*+={\circ}\ar@{}+<1pt,-10pt>*{_{l_2}}\ar@{-}@/^1pc/^{D^{-1}}[llll]&&
*+={\circ}\ar@{-}@/^1pc/^{C_2}[ll]
}
}
$$

\vspace{2mm}

We know that $D_1= DtD^{-1}C_2$ is less or equal to $v$. Note also that $v= D F^{\fty}$ cannot coincide with $D_1$ for
the whole length of the latter - for instance $v$ cannot include $\eps^{-1}\pi$ at the right hand end of $t$ on the
above diagram, for otherwise $\eps^{-1}\pi$ would be a substring of $F$, which contradicts Fact \ref{band}(i).

It follows that $D_1$ is different from $v$ (and hence less than $v$) somewhere before the first letter $\pi$ of
$D^{-1}$.

Consider the string module $L'= M(C_1^{-1}.D')$ which is obtained from $L$ by factoring out the (marked) image of the
rightmost $\eps$ in the diagram and everything to the right of that, and let $l'_1$ be the image of $l_1$ in this
factor module. From what we have said above it follows that $D'\leq v$ and therefore the formula $(C_1^{-1}.D')$ is
in $p^+$. Because (by the construction) this formula implies $\chi$, it follows that we are in case 1) and $\chi \in p^+$.
\end{proof}

It follows from the proof of this theorem that, if $M$ is any pure injective module with a homogeneous element $m$
with word $w$, then $M_w$ is a direct summand of $M$.  Thus the second part of Harland's theorem, Fact \ref{har},
has  now been justified and we have a self-contained proof of Theorem \ref{sup}.

We draw one more corollary from the proof of Theorem \ref{unique}.

\begin{cor}\label{basis}
Suppose that $w=u^{-1}v$ is a 2-sided non-periodic string over a domestic string algebra and $M_w$ is the corresponding
indecomposable pure injective module. A basis of open sets in the Ziegler topology for $M_w$ is given by the
pairs $(C^{-1}.D) \,\, / \,\, (E^{-1}.D)+ (C^{-1}.F)$, where $C\leq u< E$ in $\wh H_{-1}$ and $D\leq v< F$ in $\wh H_1$.
\end{cor}


\begin{thebibliography}{99}


\bibitem{BarPre} S.~Baratella, M.~Prest, Pure--injective modules over the dihedral algebras,
Comm. Algebra, 25(1) (1997), 11--31.

\bibitem{B-P} K.~Burke, M.~Prest, The Ziegler and Zariski spectra of some domestic string algebras,
Algebras Repr. Theory, 5 (2002), 211--234.

\bibitem{B-R} M.C.R.~Butler, C.M.~Ringel, Auslander--Reiten sequences with few middle terms and
applications to string algebras, Comm. Algebra, 15 (1987), 145--179.

\bibitem{CB}  W.~Crawley-Boevey,  Infinite-dimensional modules in the representation theory of finite-dimensional
algebras, pp. 29--54 in: I.~Reiten, S.~Smal\o\ and \O.~Solberg (Eds.), Algebras and Modules I, Canadian Math.
Soc. Conf. Proc., Vol. 23, Amer. Math. Soc, 1998.

\bibitem{Har} R.J.~Harland, Pure injective modules over tubular algebras and string algebras, PhD Thesis,
University of Manchester, 2011.

\bibitem{Kra91} H.~Krause, Maps between tree and band modules, J. Algebra, 137 (1991), 186--194.

\bibitem{Pre98} M.~Prest, Ziegler spectra of tame hereditary algebras, J. Algebra, 207 (1998), 146--164.

\bibitem{Preb2} M.~Prest, Purity, Spectra and Localization, Encyclopedia of Mathematics and its
Applications, Vol.~121, Cambridge University Press, 2009.

\bibitem{P-P04} M.~Prest, G.~Puninski, One-directed indecomposable pure injective modules over string algebras,
Colloq. Math., 101 (2004), 89--112.

\bibitem{Pun04} G.~Puninski, Super decomposable pure injective modules exist over some string algebras,
Proc. Amer. Math. Soc., 132 (2004), 1891--1898.

\bibitem{Pun07} G.~Puninski, Band combinatorics of domestic string algebras, Colloq. Math., 108 (2007), 285--296.

\bibitem{Pun12} G.~Puninski, Krull--Gabriel dimension and Cantor--Bendixson rank of 1-domestic string algebras,
Colloq. Math., 127 (2012), 185--210.

\bibitem{Pun14} G.~Puninski, Pure injective indecomposable modules over 1-domestic string algebras,
Algebras Repres. Theory, 17 (2014), 643--673.

%\bibitem{Rin75} C.M.~Ringel, The indecomposable representations of the dihedral 2-groups, Math. Ann.,
%214 (1975), 19--34.

\bibitem{Rin95} C.M.~Ringel, Some algebraically compact modules. I, pp.~419--439 in: Abelian Groups and
Modules, eds. A.~Facchini and C.~Menini, Kluwer, 1995.

\bibitem{Rin98} C.M.~Ringel, The Ziegler spectrum of a tame hereditary algebra, Colloq. Math.,
76 (1998), 105--115.

\bibitem{Rin00} C.M.~Ringel, Infinite length modules. Some examples as introduction, pp.~1--73 in:
Infinite Length Modules, eds. H.~Krause and C.M.~Ringel, Birh\"auser, 2000.

\bibitem{Sch97} J.~Schr\"oer, Hammocks for string algebras, PhD Thesis, Bielefeld, SFB 343 E97--010, 1997.

\bibitem{Zie}  M.~Ziegler,  Model theory of modules,  Ann. Pure Appl. Logic, 26(2) (1984), 149--213.

\end{thebibliography}
\end{document}